\documentclass[reqno,a4paper,11pt]{amsart}
\usepackage{amsmath, amsthm, amscd, amsfonts, amssymb, graphicx, color}
\usepackage{amssymb}
\usepackage[english]{babel}
\usepackage[dvips, lmargin=4cm, rmargin=4cm, tmargin=4cm, bmargin=4cm]{geometry}
\usepackage[colorlinks=true,urlcolor=blue,linkcolor=blue, citecolor=red]{hyperref}
\usepackage{hyperref}
\usepackage[T1]{fontenc}
\usepackage{amssymb}
\usepackage{tikz}
\usetikzlibrary{backgrounds}
\usepackage[english]{babel}
\usepackage{enumitem}
\usepackage{dsfont}
\newtheorem{thm}{Theorem}[section]
\newtheorem{defini}{Definition}[section]
\newtheorem{rem}{Remark}[section]
\newtheorem{lem}{Lemma}[section]
\newtheorem{prop}{Proposition}[section]

\numberwithin{equation}{section}

\def \R {\mathbb{R} }

\begin{document}
\title[Nonlocal double phase Neumann and Robin problem  ]{Nonlocal double phase Neumann and Robin problem with  variable $s(\cdot,\cdot)-$order}
\author[  Mohammed SRATI]
{Mohammed SRATI}
\address{Mohammed SRATI\newline
 High School of Education and Formation (ESEF), University Mohammed First, Oujda, Morocco.}
\email{srati93@gmail.com}
\maketitle
\begin{abstract}
 In this paper, we develop some properties of the $a_{x,y}(\cdot)$-Neumann derivative for the nonlocal   $s(\cdot,\cdot)$-order  operator in fractional Musielak-Sobolev spaces with  variable $s(\cdot,\cdot)-$order. Therefore we prove the basic proprieties of the correspondent function spaces.  In the second part of this paper, by means of Ekeland's variational principal and direct variational approach,  we  prove  the existence of  weak solutions to the following double phase Neumann and Robin problem  with variable $s(\cdot,\cdot)-$order :
  $$
    \left\{ 
     \begin{array}{clclc}
  (-\Delta)^{s_1(x,\cdot)}_{a^1_{(x,\cdot)}} u+(-\Delta)^{s_2(x,\cdot)}_{a^2_{(x,\cdot)}} u +\widehat{a}^1_x(|u|)u+\widehat{a}^2_x(|u|)u & = & \lambda f(x,u)    \text{ in } \Omega, \\\\
      \mathcal{N}^{s_1(x,\cdot)}_{a^1(x,\cdot)}u+\mathcal{N}^{s_2(x,\cdot)}_{a^2(x,\cdot)}u+\beta(x)\left( \widehat{a}^1_x(|u|)u+\widehat{a}^2_x(|u|)u \right)  & = & 0 \hspace*{0.2cm}   \text{ in }  \R^N\setminus \Omega,
     \end{array}
     \right. 
  $$
   where $(-\Delta)^{s_i(x,\cdot)}_{a^i_{(x,\cdot)}}$ and $\mathcal{N}^{s_i(x,\cdot)}_{a^i(x,\cdot)}$ denote the variable $s_i(\cdot,\cdot)$-order fractional Laplace operator and the nonlocal normal $a_i(\cdot,\cdot)$-derivative of $s_i(\cdot,\cdot)$-order, respectively. 
\end{abstract}
\tableofcontents
\section{Introduction}\label{S1}
The idea of a variable-order fractional derivative is an interesting extension of classical fractional derivatives. In this case, the derivative order \(\alpha\) is not constant but depends on the position (and possibly time). This approach allows modeling phenomena where local diffusion or memory properties change throughout the domain.
The variable-order fractional derivative can be defined by adapting the classical definitions of fractional derivatives. If \(\alpha = \alpha(x)\) is a function of \(x\), we can define the variable-order fractional derivative in several ways. As a result, we can define fractional spaces with a variable order of derivation. In this sense Biswas et al \cite{s1} defined the fractional Sobolev space with a fractional exponent with a variable order, subsequently Srati in \cite{srati} presented an extension of these results, namely the fractional Musielak Sobolev space with  variable $s(\cdot,\cdot)-$order, also he defined the nonlocal   $s(\cdot,\cdot)$-order  operator of elliptic type defined as follows
    {\small  $$
               \begin{aligned}
               (-\Delta)^{s(x,\cdot)}_{a_{(x,\cdot)}}u(x)=2\lim\limits_{\varepsilon\searrow 0} \int_{\R^N\setminus B_\varepsilon(x)} a_{(x,y)}\left( \dfrac{|u(x)-u(y)|}{|x-y|^{s(x,y)} }\right)\dfrac{u(x)-u(y)}{|x-y|^{s(x,y)}} \dfrac{dy}{|x-y|^{N+s(x,y)} }
               \end{aligned}
                $$}  
  for all $x\in \R^N$, where:\\
  $\bullet$ $s(\cdot,\cdot)~ : \overline{\Omega}\times\overline{\Omega}\rightarrow (0,1)$ is a continuous function such that:
  \begin{equation}
  s(x,y)=s(y,x)~~ \forall x,y \in \overline{\Omega}\times\overline{\Omega},
  \end{equation}
  \begin{equation}
  0<s^-=\inf\limits_{\overline{\Omega}\times\overline{\Omega}}s(x,y)\leqslant s^+=\sup\limits_{\overline{\Omega}\times\overline{\Omega}}s(x,y)<1.
  \end{equation}
   $\bullet$ $a_{(x,y)}(t):=a(x,y,t) : \overline{\Omega}\times\overline{\Omega}\times \R\longrightarrow \R$   is symmetric function :
 \begin{equation}\label{n4}
 a(x,y,t)=a(y,x,t) ~~ \forall(x,y,t)\in \overline{\Omega}\times\overline{\Omega}\times \R,\end{equation}

 In this work, we continue the study of the class of fractional problems in the new space $W^{s(x,y)}L_{\varPhi_{x,y}}(\Omega)$. Our main objective is to investigate, for the first time, a fractional problem involving the nonlocal $s(\cdot,\cdot)$-order elliptic operator with nonlocal Neumann and Robin boundary conditions.\\

The Neumann boundary condition, credited to the German mathematician Neumann, is also known as the boundary condition of the second kind. In this type of boundary condition, the value of the gradient of the dependent variable normal to the boundary, $\frac{\partial \phi}{\partial n}$, is prescribed on the boundary. \\

In the last years, great attention has been devoted to the study of nonlocal problems with fractional Neumann boundary condition, In this contex, Dipierro, Ros-Oton, and Valdinoci, in \cite{N5}  introduce an extension for the classical Neumann condition $\frac{\partial \phi}{\partial n} = 0$ on $\partial\Omega$ consists in the nonlocal prescription

\begin{equation}\label{n1}
     \begin{aligned}
     \mathcal{N}^s_2u(x)= \int_{\Omega}  \dfrac{u(x)-u(y)}{|x-y|^{N+2s}}dy ,~~\forall x\in \R^N\setminus \Omega.
               \end{aligned}
\end{equation}
Other Neumann problems for the fractional Laplacian (or other nonlocal operators) were introduced in \cite{N1,N2,N3}. All these different Neumann problems for nonlocal operators recover the classical Neumann problem as a limit case, and most of them have clear probabilistic interpretations as well. 
An advantage of this approach (\ref{n1}) is that the problem has a variational structure. 
In \cite{N4}, Mugnai and Proietti Lippi  introduced an extension of (\ref{n1}) for $p\neq 2$.\\

 Bahrouni et al in \cite{N6}  introduced the nonlocal Neumann and Robin boundary conditions corresponding to the fractional  $p(\cdot,\cdot)$-Laplacian as follows:
\begin{equation}\label{n3}
      \begin{aligned}
      \mathcal{N}^s_{p(x,\cdot)}u(x)= \int_{\Omega}  \dfrac{|u(x)-u(y)|^{p(x,y)-2}(u(x)-u(y))}{|x-y|^{N+sp(x,y)} }dy,~~\forall x\in \R^N\setminus \Omega,
                \end{aligned}
\end{equation}
where $p: \R^{2N} \longrightarrow (1, +\infty)$ is a symmetric, continuous function bounded and  $p(\cdot) = p(\cdot,\cdot)$. $\mathcal{N}^s_{p(x,\cdot)}$   is the nonlocal normal $p(\cdot,\cdot)$-derivative [or $p(\cdot,\cdot)$-Neumann boundary condition] and describes the natural Neumann boundary condition in the presence of the fractional $p(\cdot,\cdot)$-Laplacian, (\ref{n3}) extends the notion of the nonlocal normal derivative  for the fractional $p$-Laplacian. See also \cite{kamali}.\\

On other extention  of $p$-Neumann boundary condition, has proposed by Bahrouni and Salort in \cite{N7} as following
{\small$$
     \begin{aligned}
     \mathcal{N}^s_{a(\cdot)}u(x)= \int_{\Omega} a\left( \dfrac{|u(x)-u(y)|}{|x-y|^s }\right)\dfrac{u(x)-u(y)}{|x-y|^s} \dfrac{dy}{|x-y|^{N+s}},~~\forall x\in \R^N\setminus \Omega,
               \end{aligned}
                $$ }
where $a = A'$ such that $A$ is a Young function   and $s \in (0, 1)$. To read other problems driven by a nonlocal operator of the elliptic type in fractional Orlicz-Sobolev spaces, we refer the reader to the works \cite{SRN1,SRN2,SRT}. \\

Very recently, Srati et al. \cite{srati2},  introduced the natural Neumann boundary condition in the presence of the fractional $a_{x,y}(\cdot)$-Laplacian in fractional Musielak Sobolev spaces as following 
{\small$$
     \begin{aligned}
     \mathcal{N}^s_{a(x,\cdot)}u(x)= \int_{\Omega} a_{(x,y)}\left( \dfrac{|u(x)-u(y)|}{|x-y|^s }\right)\dfrac{u(x)-u(y)}{|x-y|^s} \dfrac{dy}{|x-y|^{N+s}},~~\forall x\in \R^N\setminus \Omega,
               \end{aligned}
                $$ }
                  denotes $a_{(x,\cdot)}-$Neumann boundary condition and present the natural Neumann boundary condition for $(-\Delta)^s_{a_{(x,\cdot)}}$ in fractional Musielak-Sobolev space. \\
                  
For the fractional derivative case of variable order, Biswas, Bahrouni and  Carvalho in \cite{x1}  introduced the nonlocal Neumann and Robin boundary conditions corresponding to the fractional  $s(\cdot,\cdot) \& p(\cdot,\cdot)$-Laplacian as follows:
\begin{equation}\label{n33}
      \begin{aligned}
      \mathcal{N}^{s(x,\cdot)}_{p(x,\cdot)}u(x)= \int_{\Omega}  \dfrac{|u(x)-u(y)|^{p(x,y)-2}(u(x)-u(y))}{|x-y|^{N+s(x,y)p(x,y)}}dy,~~\forall x\in \R^N\setminus \Omega,
                \end{aligned}
\end{equation}
where $p(\cdot,\cdot): \R^{2N} \longrightarrow (1, +\infty)$  and $s(\cdot,\cdot)~ : \overline{\Omega}\times\overline{\Omega}\rightarrow (0,1)$ are a continuous and symmetric functions.\\

 The use of fractional Neumann conditions with variable order in modular spaces, such as Orlicz spaces or Lebesgue spaces with variable exponent, extends the flexibility and applicability of mathematical models to more complex and realistic scenarios. Indeed in Orlicz spaces, this makes it possible to model data with growth conditions which differ from the polynomial growth dictated by $L_p$ spaces. For example, Orlicz spaces can handle exponential growth, which is useful in many applications like nonlinear elasticity, fluid dynamics, and image processing. In Variable Exponent Lebesgue spaces: These spaces generalize classical Lebesgue spaces by allowing the exponent $t^p$ to vary with the position $x$. This flexibility is crucial in problems where the regularity of the solution may change across the domain, such as in materials with spatially varying properties or in image processing where different regions of the image may require different levels of smoothing.\\

For that, in this paper, we introduce the fractional Neumann boundary condition   with variable order in the presence of the  nonlocal $s(\cdot,\cdot)$-order elliptic operator. Therefore  we are concerned with the existence of  weak solutions to the following Neumann-Robin problem
 $$\label{P}
 (\mathcal{P}_a)  \left\{ 
    \begin{array}{clclc}
 (-\Delta)^{s_1(x,\cdot)}_{a^1_{(x,\cdot)}} u+(-\Delta)^{s_2(x,\cdot)}_{a^2_{(x,\cdot)}} u +\widehat{a}^1_x(|u|)u+\widehat{a}^2_x(|u|)u & = & \lambda f(x,u)    \text{ in } \Omega, \\\\
     \mathcal{N}^{s_1(x,\cdot)}_{a^1(x,\cdot)}u+\mathcal{N}^{s_2(x,\cdot)}_{a^2(x,\cdot)}u+\beta(x)\left( \widehat{a}^1_x(|u|)u+\widehat{a}^2_x(|u|)u \right)  & = & 0 \hspace*{0.2cm}   \text{ in }  \R^N\setminus \Omega,
    \end{array}
    \right. 
 $$
  where $\Omega$ is an open bounded subset in $\R^N$, $N\geqslant 1$,   with Lipschitz boundary $\partial \Omega$, 
 $f: \Omega\times \R \longrightarrow \R$ is a Carath\'eodory function, $\beta\in L^{\infty}(\R^N\setminus \Omega)$ such that $\beta\geqslant 0$ in $\R^N\setminus \Omega$ and  $(-\Delta)^{s_i(x,\cdot)}_{a^i_{(x,\cdot)}}$ $(i = 1, 2)$  are the nonlocal integro-differential operators of elliptic type defined as follows
    {\small  $$
               \begin{aligned}
               (-\Delta)^{s_i(x,\cdot)}_{a^i_{(x,\cdot)}}u(x)=2\lim\limits_{\varepsilon\searrow 0} \int_{\R^N\setminus B_\varepsilon(x)} a^i_{(x,y)}\left( \dfrac{|u(x)-u(y)|}{|x-y|^{s_i(x,y)} }\right)\dfrac{u(x)-u(y)}{|x-y|^{s_i(x,y)}} \dfrac{dy}{|x-y|^{N+s_i(x,y)} },
               \end{aligned}
                $$}  
  for all $x\in \R^N$, where:\\
    $\bullet$ $s_i(\cdot,\cdot)~ : \overline{\Omega}\times\overline{\Omega}\rightarrow (0,1)$ is a continuous function such that:
    \begin{equation}
    s_i(x,y)=s_i(y,x),~~ \forall x,y \in \overline{\Omega}\times\overline{\Omega},
    \end{equation}
    \begin{equation}
    0<s_i^-=\inf\limits_{\overline{\Omega}\times\overline{\Omega}}s_i(x,y)\leqslant s_i^+=\sup\limits_{\overline{\Omega}\times\overline{\Omega}}s_i(x,y)<1.
    \end{equation}
     $\bullet$  $(x,y,t)\mapsto a^i_{(x,y)}(t):=a^i(x,y,t) : \overline{\Omega}\times\overline{\Omega}\times \R\longrightarrow \R$ $(i = 1, 2)$  are symmetric functions :
 \begin{equation}\label{n4}
 a^i(x,y,t)=a^i(y,x,t) ~~ \forall(x,y,t)\in \overline{\Omega}\times\overline{\Omega}\times \R,\end{equation}
   and the functions : $\varphi^i(\cdot,\cdot,\cdot) : \overline{\Omega}\times\overline{\Omega}\times \R \longrightarrow \R$ $(i = 1, 2)$ defined by  
$$
  \varphi^i_{x,y}(t):=\varphi^i(x,y,t)= \left\{ 
          \begin{array}{clclc}
        a^i(x,y,|t|)t   & \text{ for }& t\neq 0, \\\\
          0  & \text{ for } & t=0,
          \end{array}
          \right. 
$$
are increasing homeomorphisms from $\R$ onto itself. For $i = 1, 2$, let 
$$\varPhi^i_{x,y}(t):=\varPhi^i(x,y,t)=\int_{0}^{t}\varphi^i_{x,y}(\tau)d\tau~~\text{ for all } (x,y)\in \overline{\Omega}\times\overline{\Omega},~~\text{ and all } t\geqslant 0.$$  
Then, $\varPhi^i_{x,y}$ are a Musielak functions (see \cite{mu}).

Also, we take $ \widehat{a}^i_x(t):=\widehat{a}^i(x,t)=a^i_{(x,x)}(t)  ~~ \forall~ (x,t)\in \overline{\Omega}\times \R$ $(i = 1, 2)$. Then the functions $\widehat{\varphi}^i(\cdot,\cdot) : \overline{\Omega}\times \R \longrightarrow \R$ defined  by :
  $$
     \widehat{\varphi}^i_{x}(t):=\widehat{\varphi}^i(x,t)= \left\{ 
          \begin{array}{clclc}
        \widehat{a}^i(x,|t|)t   & \text{ for }& t\neq 0, \\\\
          0  & \text{ for } & t=0,
          \end{array}
          \right. 
       $$
are increasing homeomorphisms from $\R$ onto itself. If we set 
\begin{equation}\label{phi}
\widehat{\varPhi}^i_{x}(t):=\widehat{\varPhi}^i(x,t)=\int_{0}^{t}\widehat{\varphi}^i_{x}(\tau)d\tau ~~\text{ for all}~~ t\geqslant 0.
\end{equation}  
Then, $\widehat{\varPhi}^i_{x}$ is also a Musielak function.\\
 
Furthermore, $\mathcal{N}^{s_i(x,\cdot)}_{a^i_{(x,\cdot)}}$ $(i = 1, 2)$ are defined by
{\small$$
     \begin{aligned}
     \mathcal{N}^{s_i(x,\cdot)}_{a^i{(x,\cdot)}}u(x)= \int_{\Omega} a^i_{(x,y)}\left( \dfrac{|u(x)-u(y)|}{|x-y|^{s_i(x,y)} }\right)\dfrac{u(x)-u(y)}{|x-y|^{s_i(x,y)}} \dfrac{dy}{|x-y|^{N+s_i(x,y)}},~~\forall x\in \R^N\setminus \Omega,
               \end{aligned}
                $$ }
                  denotes $a^i_{(x,\cdot)}-$Neumann boundary condition and present the natural Neumann boundary condition for $(-\Delta)^{s_i(x,\cdot)}_{a^i_{(x,\cdot)}}$ in fractional Musielak-Sobolev space.

       This paper is organized as follows, In Section \ref{S1}, we  set  the problem  \hyperref[P]{$(\mathcal{P}_{a})$} and the related hypotheses. Moreover,  we are introduced the new  Neumann boundary condition  associated  to nonlocal $s(\cdot,\cdot)$-order elliptic operator.  The Section \ref{S2}, is devoted to recall
   some properties of $s(x,y)$-fractional Musielak-Sobolev spaces. In section \ref{S3}, we introduce the corresponding function space for weak solutions of \hyperref[P]{$(\mathcal{P}_{a})$}, and we prove some properties, and state
   the corresponding Green formula for problems such as \hyperref[P]{$(\mathcal{P}_{a})$}. In section \ref{S4},    by means of Ekeland's variational principle and direct variational approach,
       we obtain the existence  of  a nontrivial weak solution for our Problem.    
        \section{Preliminaries results}\label{S2}                      
To deal with this situation we define the fractional Musielak-Sobolev space with  variable $s(\cdot,\cdot)-$order to
investigate Problem \hyperref[P]{$(\mathcal{P}_{a})$}. Let us recall the definitions and some elementary
properties of this spaces. We refer the reader to \cite{benkirane,benkirane2,3,SRH,srati} for further reference
and for some of the proofs of the results in this section.

 For the function $\widehat{\varPhi}_x$ $(i=1,2)$ given in (\ref{phi}), we introduce the Musielak space as follows
  $$L_{\widehat{\varPhi}^i_x} (\Omega)=\left\lbrace u : \Omega \longrightarrow \R \text{ mesurable }: \int_\Omega\widehat{\varPhi}^i_x(\lambda |u(x)|)dx < \infty \text{ for some } \lambda>0 \right\rbrace. $$
The space $L_{\widehat{\varPhi}^i_x} (\Omega)$ is a Banach space endowed with the Luxemburg norm 
$$||u||_{\widehat{\varPhi}^i_x}=\inf\left\lbrace \lambda>0 \text{ : }\int_\Omega\widehat{\varPhi}^i_x\left( \dfrac{|u(x)|}{\lambda}\right) dx\leqslant 1\right\rbrace. $$
 The conjugate function of $\varPhi^i_{x,y}$ $(i=1,2)$ is defined by $\overline{\varPhi^i}_{x,y}(t)=\int_{0}^{t}\overline{\varphi^i}_{x,y}(\tau)d\tau$ $\text{ for all } (x,y)\in\overline{\Omega}\times\overline{\Omega}$  $\text{ and all } t\geqslant 0$, where $\overline{\varphi^i}_{x,y} : \R\longrightarrow \R$ is given by $\overline{\varphi^i}_{x,y}(t):=\overline{\varphi^i}(x,y,t)=\sup\left\lbrace \alpha \text{ : } \varphi^i(x,y,\alpha)\leqslant t\right\rbrace.$ Furthermore, we have the following H\"older type inequality
  \begin{equation}
   \left| \int_{\Omega}uvdx\right| \leqslant 2||u||_{\widehat{\varPhi}^i_x}||v||_{\overline{\widehat{\varPhi}^i}_x}\hspace*{0.5cm} \text{ for all } u \in L_{\widehat{\varPhi}^i_x}(\Omega)  \text{ and } v\in L_{\overline{\widehat{\varPhi}^i}_x}(\Omega).
   \end{equation}
    Throughout this paper, for $i=1,2$, we assume that there exist two positive constants ${\varphi_i}^+$ and $\varphi_i^-$ such that 
\begin{equation}\label{v1}\tag{$\varPhi_1$}
    1<\varphi_i^-\leqslant\dfrac{t\varphi^i_{x,y}(t)}{\varPhi^i_{x,y}(t)}\leqslant \varphi_i^+<+\infty\text{ for all } (x,y)\in\overline{\Omega}\times\overline{\Omega}~~\text{ and all } t\geqslant 0. \end{equation}
    This relation implies  that
    \begin{equation}\label{A2}
        1<\varphi_i^-\leqslant \dfrac{t\widehat{\varphi}^i_{x}(t)}{\widehat{\varPhi}^i_{x}(t)}\leqslant\varphi_i^+<+\infty,\text{ for all } x\in\overline{\Omega}~~\text{ and all } t\geqslant 0.\end{equation}
             It follows that  $\varPhi^i_{x,y}$ and $\widehat{\varPhi}^i_{x}$ satisfy the global $\Delta_2$-condition (see \cite{ra}), written $\varPhi^i_{x,y}\in \Delta_2$ and $\widehat{\varPhi}^i_{x}\in \Delta_2$, that is,
    \begin{equation}\label{r1}
    \varPhi^i_{x,y}(2t)\leqslant K_1\varPhi^i_{x,y}(t)~~ \text{ for all } (x,y)\in\overline{\Omega}\times\overline{\Omega},~~\text{ and  all } t\geqslant 0,
    \end{equation} and
    \begin{equation}\label{rr1}
        \widehat{\varPhi}^i_{x}(2t)\leqslant K_2\widehat{\varPhi}^i_{x}(t) ~~\text{ for any } x\in\overline{\Omega},~~\text{ and  all } t\geqslant 0,
        \end{equation}
 where $K_1$ and $K_2$ are two positive constants. 
 
 Furthermore, for $i=1,2$, we assume that $\varPhi^i_{x,y}$ satisfies the following condition
  \begin{equation}\label{f2.}\tag{$\varPhi_2$}
  \text{ the function } [0, \infty) \ni t\mapsto \varPhi^i_{x,y}(\sqrt{t}) \text{ is convex. }
  \end{equation}
 

   Now, due to the nonlocality of the operators $(-\Delta)^{s_i(x,y)}_{a_{(x,\cdot)}}$ $(i = 1, 2)$,  we  define the new $s(x,y)$-fractional Musielak-Sobolev spaces as introduce in \cite{srati} as follows 
    \begingroup\makeatletter\def\f@size{9}\check@mathfonts$$ W^{s_i(x,y)}{L_{\varPhi^i_{x,y}}}(\Omega)=\Bigg\{u\in L_{\widehat{\varPhi^i}_x}(\Omega) :  \int_{\Omega} \int_{\Omega} \varPhi^i_{x,y}\left( \dfrac{\lambda| u(x)- u(y)|}{|x-y|^{s_i(x,y)}}\right) \dfrac{dxdy}{|x-y|^N}< \infty \text{ for some } \lambda >0 \Bigg\}.
$$\endgroup
This space can be equipped with the norm
\begin{equation}\label{r2}
||u||_{s_i(x,y),\varPhi^i_{x,y}}=||u||_{\widehat{\varPhi}^i_x}+[u]_{s_i(x,y),\varPhi^i_{x,y}},
\end{equation}
where $[\cdot]_{s_i(x,y),\varPhi^i_{x,y}}$ is the Gagliardo seminorm defined by 
$$[u]_{s_i(x,y),\varPhi^i_{x,y}}=\inf \Bigg\{\lambda >0 :  \int_{\Omega} \int_{\Omega} \varPhi^i_{x,y}\left( \dfrac{|u(x)- u(y)|}{\lambda|x-y|^{s_i(x,y)}}\right) \dfrac{dxdy}{|x-y|^N}\leqslant 1 \Bigg\}.
$$

\begin{thm}$($\cite{srati}$)$.
      Let $\Omega$ be an open subset of $\R^N$. For $i=1,2$, the spaces $W^{s_i(x,y)}L_{\varPhi^i_{x,y}}(\Omega)$ $(i = 1, 2)$ are a Banach spaces with respect to the norm $(\ref{r2})$, and a  separable $($resp. reflexive$)$ spaces if and only if $\varPhi^i_{x,y} \in \Delta_2$ $($resp. $\varPhi^i_{x,y}\in \Delta_2 $ and $\overline{\varPhi^i}_{x,y}\in \Delta_2$$)$. Furthermore,
if   $\varPhi^i_{x,y} \in \Delta_2$ and $\varPhi^i_{x,y}(\sqrt{t})$ is convex, then  the spaces $W^{s_i(x,y)}L_{\varPhi^i_{x,y}}(\Omega)$ are an uniformly convex space.
      \end{thm}

           \begin{defini}$($\cite{benkirane}$)$.
           We say that $\varPhi^i_{x,y}$ $(i=1,2)$ satisfies the fractional boundedness condition, written $\varPhi^i_{x,y}\in \mathcal{B}_{f}$, if
         \begin{equation}\tag{$\varPhi_3$}
        \label{v3}         
           \sup\limits_{(x,y)\in \overline{\Omega}\times\overline{\Omega}}\varPhi^i_{x,y}(1)<\infty.  \end{equation}
           \end{defini}
           \begin{thm}  $($\cite{benkirane,srati}$)$.    \label{TT}
                       Let $\Omega$ be an open subset of $\R^N$. Assume that  $\varPhi^i_{x,y}\in \mathcal{B}_{f}$ $(i=1,2)$. 
                       Then,
                       $$C^2_0(\Omega)\subset W^{s_i(x,y)}L_{\varPhi^i_{x,y}}(\Omega).$$
                  \end{thm}

   
               
   For $i=1,2$  we denote by $\widehat{\varPhi^i}_{x}^{-1}$ the inverse function of $\widehat{\varPhi^i}_{x}$  which satisfies the following conditions:
           \begin{equation}\label{15}
           \int_{0}^{1} \dfrac{\widehat{\varPhi^i}_{x}^{-1}(\tau)}{\tau^{\frac{N+s_i^-}{N}}}d\tau<\infty~~ \text{ for all } x\in \overline{\Omega} \text{ and } ~~i=1,2,
           \end{equation}
           
           \begin{equation}\label{16n}
           \int_{1}^{\infty} \dfrac{\widehat{\varPhi^i}_{x}^{-1}(\tau)}{\tau^{\frac{N+s_i^-}{N}}}d\tau=\infty ~~\text{ for all }x\in \overline{\Omega} \text{ and } ~~i=1,2.
           \end{equation}
           If (\ref{16n}) is satisfied, we define the inverse  Musielak conjugate function of $\widehat{\varPhi}_x$ as follows
           \begin{equation}\label{17}
           (\widehat{\varPhi}^*_{i})^{-1}(t)=\int_{0}^{t}\dfrac{\widehat{\varPhi^i}_{x}^{-1}(\tau)}{\tau^{\frac{N+s_i^-}{N}}}d\tau ~~\text{ for } ~~i=1,2.
           \end{equation}
            \begin{thm}\cite{srati}\label{th2.}
          Let $\Omega$  be a bounded open
           subset of  $\R^N$ with $C^{0,1}$-regularity 
             and bounded boundary. If $(\ref{15})$ and  $(\ref{16n})$  hold, then 
          \begin{equation}\label{18}
           W^{s_i(x,y)}{L_{\varPhi^i_{x,y}}}(\Omega)\hookrightarrow L_ {\widehat{\varPhi}^*_{i}}(\Omega)~~\text{ for } i=1,2.
          \end{equation}
         Moreover, the embedding
                    \begin{equation}\label{27}
                     W^{s_i(x,y)}{L_{\varPhi^i_{x,y}}}(\Omega)\hookrightarrow L_{B_x}(\Omega)~~\text{ for } i=1,2,
                    \end{equation}
                    is compact for all $B_x\prec\prec \widehat{\varPhi}^*_{i}$ $~~\text{ for } ~~i=1,2.$
                    \end{thm}           
 
        Finally, the proof of our existence result is based on the following Ekeland's variational principle theorem and direct variational approach.
   \begin{thm}\label{th1}(\cite{ek})
   Let V be a complete metric space and $F : V \longrightarrow \R\cup \left\lbrace +\infty\right\rbrace$ be a lower semicontinuous functional on $V$, that is bounded below and not identically equal to $+\infty$. Fix $\varepsilon>0$ and a  point $u\in V$ 
     such that
    $$F(u)\leqslant \varepsilon +\inf\limits_{x\in V}F(x).$$ Then for every $\gamma > 0$,
     there exists some point $v\in V$ such that :
     $$F(v)\leqslant F(u),$$
     $$d(u,v)\leqslant \gamma$$
     and for all $w\neq v$
     $$F(w)> F(v)-\dfrac{\varepsilon}{\gamma}d(v,w).$$
   \end{thm}
   
    \begin{thm}\label{th2} (\cite{110})
        Suppose that $Y$ is a reflexive Banach space with norm $||.||$ and let
        $V\subset Y$ be a weakly closed subset of $Y$. Suppose $E : V \longrightarrow \R \cup \left\lbrace +\infty\right\rbrace $ is coercive
        and (sequentially) weakly lower semi-continuous on $V$ with respect to $Y$, that
        is, suppose the following conditions are fulfilled:
     \begin{itemize}
        \item[$\bullet$] $E(u)\rightarrow \infty$ as $||u||\rightarrow \infty$, $u\in V$.
        
        \item[$\bullet$]  For any $u\in V$, any sequence $\left\lbrace u_n\right\rbrace $ in $V$ such that $u_n\rightharpoonup u$ weakly in $X$
        there holds:
       $$E(u)\leqslant \liminf_{n\rightarrow \infty}E(u_n).$$
         \end{itemize} 
        Then $E$ is bounded from below on $V$ and attains its infimum  in $V$.     
         \end{thm}        
\section{Some qualitative properties of  $\mathcal{N}^{s(x,\cdot)}_{a(x,\cdot)}$}\label{S3}
The aim of this section is to give the basic properties of the fractional $a^i(x,y)$-Laplacian with the associated $a^i(x,y)$-Neumann boundary
condition.\\
          
Let $u : \R^N \longrightarrow \R$ be a measurable function, for $i=1,2$ we set
$$\|u\|_{X_i}=[u]_{s_i(x,y),\varPhi^i_{x,y}}+\|u\|_{\widehat{\varPhi}^i_x}+\|u\|_{\widehat{\varPhi}^i_x,\beta,C\Omega}$$ 
    where
{\small$$[u]_{s_i(x,y),\varPhi^i_{x,y}}=\inf \Bigg\{\lambda >0 :  \int_{\R^{2N}\setminus (C\Omega)^2} \varPhi^i_{x,y}\left( \dfrac{|u(x)- u(y)|}{\lambda|x-y|^{s_i(x,y)}}\right) \dfrac{dxdy}{|x-y|^N}\leqslant 1 \Bigg\}
$$ }   
and
$$\|u\|_{\widehat{\varPhi}^i_x,\beta,C\Omega}=\inf\left\lbrace \lambda>0 \text{ : }\int_{C\Omega}\beta(x)\widehat{\varPhi}^i_x\left( \dfrac{|u(x)|}{\lambda}\right) dx\leqslant 1\right\rbrace $$
    with $C\Omega =\R^N\setminus \Omega$. We define 
    $$X_i=\left\lbrace u : \R^N\longrightarrow \R~~\text{ measurable } : \|u\|_{X_i}<\infty\right\rbrace.$$ 
    \begin{rem}
 It is easy to see that $\|\cdot\|_{X_i}$ is a norm on $X_i$ $(i=1,2)$. We only show that if $\|u\|_{X_i}=0$, then $u=0$ a.e. in $\R^N$. Indeed, form $\|u\|_{X_i}=0$, we get $\|u\|_{\widehat{\varPhi}^i_x}=0$, which implies that 
 \begin{equation}\label{N1}
 u=0 ~~\text{a.e. in } \Omega
 \end{equation} 
 and 
 \begin{equation}\label{N2}
 \int_{\R^{2N}\setminus (C\Omega)^2} \varPhi^i_{x,y}\left( \dfrac{|u(x)- u(y)|}{|x-y|^{s_i(x,y)}}\right) \dfrac{dxdy}{|x-y|^N}=0.
 \end{equation} 
 By $(\ref{N2})$, we deduce that $u(x)=u(y)$ in $\R^{2N}\setminus (C\Omega)^2$, that is $u=c\in \R$ in $\R^N$, and by $(\ref{N1})$ we have $u=0$ a.e. in $\R^N$.
    
      \end{rem} 
    \begin{prop}\label{rem}
     Note that the norm $\|\cdot\|_{X_i}$ is equivalent on $X_i$ to
     $$\|u\|_i:=\inf\left\lbrace \lambda>0~~:~~\rho_{s_i(x,y)}\left( \dfrac{u}{\lambda}\right) \leqslant 1\right\rbrace $$
     where $(i=1,2)$, and the modular function $\rho_{s_i(x,y)}~~ : X\longrightarrow \R$ is defined by
     \begin{equation}\label{smodN}
     \begin{aligned}
     \rho_{s_i(x,y)}(u)=&\int_{\R^{2N}\setminus (C\Omega)^2} \varPhi^i_{x,y}\left( \dfrac{|u(x)- u(y)|}{|x-y|^{s_i(x,y)}}\right) \dfrac{dxdy}{|x-y|^N}\\
     &+\int_{\Omega}\widehat{\varPhi}^i_x\left(|u(x)|\right) dx+\int_{C\Omega}\beta(x)\widehat{\varPhi}^i_x\left( |u(x)|\right) dx.
      \end{aligned}
           \end{equation}    
     \end{prop}
      The proof is similar to \cite[Proposition 2.1]{benkirane}; see also \cite[Proposition 3]{sr5}.\\

 An important role in manipulating  the $s(\cdot,\cdot)$-fractional Musielak-Sobolev spaces is played by the modular function $(\ref{smodN})$. It is worth noticing that the relation between the norm and the modular shows an equivalence between the topology defined by the norm and that defined by the modular.            
    \begin{prop}\label{Nmod}
             Assume that (\ref{v1}) is satisfied. Then, for any $u \in X_i$ $(i=1,2)$, the following relations hold true:
               \begin{equation}\label{Nmod1}
         ||u||_i>1\Longrightarrow      ||u||_i^{{\varphi^i}^-} \leqslant  \rho_{s_i(x,y)}(u)\leqslant  ||u||_i^{{\varphi^i}^+},
               \end{equation}
               \begin{equation}\label{Nmod2}
                    ||u||_i<1\Longrightarrow    ||u||_i^{{\varphi^i}^+} \leqslant  \rho_{s_i(x,y)}(u)\leqslant  ||u||_i^{{\varphi^i}^-}. \end{equation}
               \end{prop}  
 The proof is similar to \cite[Proposition 2.2]{benkirane}.
\begin{prop}
$\left(X_i, \|\cdot\|_{X_i}\right)$  $(i=1,2)$ is a reflexive Banach space.
\end{prop} 
\begin{proof} 
Now, we prove that $X_i$ is complete. For this, let $\left\lbrace u_n\right\rbrace $ be a Cauchy sequence in $X_i$. In particular  $\left\lbrace u_n\right\rbrace $ is a Cauchy sequence in $L_{\widehat{\varPhi}^i_x(\Omega)}$ and so, there exists $u\in L_{\widehat{\varPhi}^i_x(\Omega)}$ such that 
$$u_n\longrightarrow u~~\text{in}~~L_{\widehat{\varPhi}^i_x(\Omega)}~~\text{and a.e. in } \Omega.$$
Then, we can find $Z_1\subset \R^N$ such that 
\begin{equation}\label{N3}
|Z_1|=0~~\text{ and } u_n(x)\longrightarrow u(x)~~\text{ for every } x\in \Omega\setminus Z_1.
\end{equation} 
For any $u : \R^N\longrightarrow \R$, and for any $(x,y)\in \R^{2N}$, we set
$$E_u(x,y)=\dfrac{(u(x)-u(y))}{|x-y|^{s_i(x,y)}}\mathcal{X}_{\R^{2N}\setminus (C\Omega)^2}(x,y).$$
Using the fact that $\left\lbrace u_n\right\rbrace $ is a Cauchy sequence in $L_{\varPhi^i_{x,y}}\left( \R^{2N},d\mu\right)$, where $\mu$   is a  measure on  $\Omega\times\Omega$ which is given by
           $d\mu :=|x-y|^{-N}dxdy.$ So, there exists a subsequence $\left\lbrace E_{u_n}\right\rbrace$ converges to $E_u$ in $L_{\varPhi^i_{x,y}}\left( \R^{2N},d\mu\right)$ and a.e. in $\R^{2N}$. Then, we can find $Z_2\subset \R^{2N}$ such that
\begin{equation}\label{N4}
|Z_2|=0~~\text{ and } E_{u_n}(x,y)\longrightarrow E_u(x,y)~~\text{ for every } (x,y)\in \R^{2N}\setminus Z_2.
\end{equation} 
For any $x\in \Omega$, we set
$$S_x:=\left\lbrace y\in \R^N~~:~~(x,y)\in \R^{2N}\setminus Z_2\right\rbrace $$
$$W:=\left\lbrace (x,y)\in \R^{2N},~~x\in \Omega~~\text{and}~~y\in \R^N\setminus S_x\right\rbrace $$
$$V:=\left\lbrace x\in \Omega~~:~~|\R^N\setminus S_x|=0\right\rbrace.$$
Let $(x,y)\in W$, we have $y\in \R^N\setminus S_x$. Then $(x,y)\notin \R^{2N}\setminus Z_2$, i.e. $(x,y)\in Z_2$. So
$$W\subset Z_2,$$
therefore, by $(\ref{N4})$
$$|W|=0,$$
then, by the Fubini's Theorem we have
$$0=|W|=\int_{\Omega}| \R^N\setminus S_x|dx,$$
which implies that $| \R^N\setminus S_x|=0$ a.e $x\in \Omega$. It follows that $|\Omega\setminus V|=0$. This end with $(\ref{N3})$, implies that
$$|\Omega\setminus (V\setminus Z_1)|=|(\Omega\setminus V)\cup Z_1|\leqslant |\Omega\setminus V|+|Z_1|=0.$$
In particular $V\setminus Z_1\neq \varnothing,$ then we can fix $x_0\in V\setminus Z_1$, and by $(\ref{N3})$, it follows
$$\lim\limits_{n\rightarrow \infty} u_n(x_0)=u(x_0).$$
In addition, since $x_0\in V,$ we obtain $|\R^N\setminus S_{x_0}|=0$. Then, for almost all $y\in \R,$ this yields $(x_0,y)\in \R^{2N}\setminus Z_2$, and hence, by
$(\ref{N4})$ 
$$\lim\limits_{n\rightarrow \infty} E_{u_n}(x_0,y)=E_u(x_0,y).$$
Since $\Omega\times C\Omega \subset \R^{2N}\setminus (C\Omega)^2$, we have 
$$E_{u_n}(x_0,y)=\dfrac{(u_n(x_0)-u_n(y))}{|x_0-y|^{s_i(x,y)}}\mathcal{X}_{\R^{2N}\setminus (C\Omega)^2}(x_0,y)$$
for almost all $y\in C\Omega.$ However, this implies
$$
\begin{aligned}
\lim\limits_{n\rightarrow \infty} u_n(y)&=\lim\limits_{n\rightarrow \infty}\left(  u_n(x_0) -|x_0-y|^{s_i(x,y)}E_{u_n}(x_0,y)\right)\\
&= u(x_0) -|x_0-y|^{s_i(x,y)}E_{u}(x_0,y)
\end{aligned}
$$
for almost all $y\in C\Omega.$
Combining this end with $(\ref{N3})$, we see that $u_n$ is converges to some $u$ a.e. in $\R^N$. Since $u_n$ is a Cauchy sequence in $X$, so for any $\varepsilon>0$, there exists $N_\varepsilon>0$ such that for any $k>N_\varepsilon$, we have by applying Fatou's Lemma
$$
\begin{aligned}
\varepsilon\geqslant & \liminf\limits_{k\rightarrow \infty}\|u_n-u_k\|_{X_i}\\
&\geqslant c\liminf\limits_{k\rightarrow \infty}\|u_n-u_k\|_i\\
&\geqslant c \liminf\limits_{k\rightarrow \infty}\left( \rho_{s_i(x,y)}(u_n-u_k)\right) ^{\frac{1}{{\varphi^i}^\pm}}\\
&\geqslant c \left( \rho_{s_i(x,y)}(u_n-u)\right) ^{\frac{1}{{\varphi^i}^\pm}}\\
&\geqslant c \|u_n-u\|_i^{\frac{{\varphi^i}^\pm}{{\varphi^i}^\pm}}\\
&\geqslant c \|u_n-u\|_{X_i}^{\frac{{\varphi^i}^\pm}{{\varphi^i}^\pm}},
\end{aligned}
$$
where $c$ is a positive constant given by Proposition $\ref{rem}$. This implies that $u_n$ converge to $u$ in $X_i$, and so $X_i$ is complete space.
Now, we show that $X_i$ is a reflexive space. For this, we consider the following space
$$Y_i=L_{\widehat{\varPhi}^i_x}(\Omega)\times L_{\widehat{\varPhi}^i_x}(C\Omega)\times L_{\widehat{\varPhi}^i_{x,y}}\left( \R^{2N}\setminus (C\Omega)^2,d\mu\right) $$
endowed with the norm
$$\|u\|_{Y_i}=[u]_{s_i(x,y),\varPhi^i_{x,y},\R^{2N}\setminus (C\Omega)^2}+\|u\|_{\widehat{\varPhi}^i_x}+\|u\|_{\widehat{\varPhi}^i_x,\beta,C\Omega}.$$
We note that $(Y_i, \|\cdot\|_{Y_i})$ is a reflexive Banach space, we consider the map 
$T : X_i\longrightarrow Y_i$ defined as :
$$T(u)=\left(u,u,D^{s_i(x,y)}u\right),$$
where   $D^{s_i(x,y)}u:=\dfrac{u(x)-u(y)}{|x-y|^{s_i(x,y)}}.$  By construction, we have that
$$\|T(u)\|_{Y_i}=\|u\|_{X_i}.$$
Hence, $T$ is an isometric from $X_i$ to the reflexive space $Y_i$. This show that $X_i$ is reflexive.
\end{proof}
\begin{prop}\label{N10}

      Let $\Omega$  be a bounded open
       subset of  $\R^N$ with $C^{0,1}$-regularity 
         and bounded boundary. If $(\ref{15})$ and  $(\ref{16n})$  hold, then 
      \begin{equation}\label{N18}
       X_i\hookrightarrow L_ {\widehat{\varPhi}^*_{i}}(\Omega)~~\text{ for } i=1,2.
      \end{equation}
              In particular, the embedding
                \begin{equation}\label{N27}
                 X_i\hookrightarrow L_{B_x}(\Omega) ~~\text{ for } i=1,2.
                \end{equation}
                is compact for all $B_x\prec\prec \widehat{\varPhi}^*_{i} \text{ for } i=1,2.$
         
\end{prop}  
 \begin{proof}
 Since $\Omega\times\Omega\subset \R^{2N}\setminus (C\Omega)^2$. Then
 $$||u||_{s_i(x,y),\varPhi^i_{x,y}}\leqslant \|u\|_{X_i}~~\text{for all }~~u\in X_i \text{ and } i=1,2.$$
 Therefore, by Theorem \ref{th2.}, we get our desired result.
 \end{proof}   
 
    Now, by integration by part formula, we have the following result.
         \begin{prop}
         Let $u\in X_i$ $(i=1,2)$, then
         $$\int_\Omega  (-\Delta)^{s_i(x,\cdot)}_{a^i_{(x,\cdot)}}u(x) dx=-\int_{\R^N\setminus \Omega}\mathcal{N}^{s_i(x,\cdot)}_{a^i(x,\cdot)}u(x)dx.$$
         \end{prop} 
         \begin{proof}  
       Since the role of $x$ and $y$ are symmetric and $a^i(\cdot,\cdot)$ and $s_i(\cdot,\cdot)$ are a symmetric functions, we obtain
       $$
       \begin{aligned}
     \int_{\Omega} \int_{\Omega} & a^i_{(x,y)}\left( \dfrac{|u(x)-u(y)|}{|x-y|^{s_i(x,y)} }\right)\dfrac{u(x)-u(y)}{|x-y|^{s_i(x,y)}} \dfrac{dxdy}{|x-y|^{N+s_i(x,y)}}\\
     &= - \int_{\Omega} \int_{\Omega} a^i_{(x,y)}\left( \dfrac{|u(x)-u(y)|}{|x-y|^{s_i(x,y)} }\right)\dfrac{u(y)-u(x)}{|x-y|^{s_i(x,y)}} \dfrac{dxdy}{|x-y|^{N+s_i(x,y)}}\\
     &=- \int_{\Omega} \int_{\Omega} a^i_{(y,x)}\left( \dfrac{|u(y)-u(x)|}{|x-y|^{s_i(y,x)} }\right)\dfrac{u(x)-u(y)}{|x-y|^{s_i(y,x)}} \dfrac{dydx}{|x-y|^{N+s_i(y,x)}}\\
     &=- \int_{\Omega} \int_{\Omega} a^i_{(x,y)}\left( \dfrac{|u(x)-u(y)|}{|x-y|^{s_i(x,y)} }\right)\dfrac{u(x)-u(y)}{|x-y|^{s_i(x,y)}} \dfrac{dxdy}{|x-y|^{N+{s_i(x,y)}}}.\\
       \end{aligned}  
       $$
       This implies that 
       $$2\int_{\Omega} \int_{\Omega}  a^i_{(x,y)}\left( \dfrac{|u(x)-u(y)|}{|x-y|^{s_i(x,y)} }\right)\dfrac{u(x)-u(y)}{|x-y|^{s_i(x,y)}} \dfrac{dxdy}{|x-y|^{N+s_i(x,y)}}=0$$
       that is,
      $$ \int_{\Omega} \int_{\Omega}  a^i_{(x,y)}\left( \dfrac{|u(x)-u(y)|}{|x-y|^{s_i(x,y)} }\right)\dfrac{u(x)-u(y)}{|x-y|^{s_i(x,y)}} \dfrac{dxdy}{|x-y|^{N+s_i(x,y)}}=0.$$
   Hence, we have that
  $$
   \begin{aligned}
   \int_\Omega  (-\Delta)^{s_i(x,\cdot)}_{a^i_{(x,\cdot)}}u(x) dx &=\int_{\Omega} \int_{\R^N}  a^i_{(x,y)}\left( \dfrac{|u(x)-u(y)|}{|x-y|^{s_i(x,y)} }\right)\dfrac{u(x)-u(y)}{|x-y|^{s_i(x,y)}} \dfrac{dydx}{|x-y|^{N+s_i(x,y)}}\\
   &= \int_{\Omega} \int_{\R^N\setminus \Omega}  a^i_{(x,y)}\left( \dfrac{|u(x)-u(y)|}{|x-y|^{s_i(x,y)} }\right)\dfrac{u(x)-u(y)}{|x-y|^{s_i(x,y)}} \dfrac{dydx}{|x-y|^{N+s_i(x,y)}}\\
   & ~~+\int_{\Omega} \int_{\Omega}  a^i_{(x,y)}\left( \dfrac{|u(x)-u(y)|}{|x-y|^{s_i(x,y)} }\right)\dfrac{u(x)-u(y)}{|x-y|^{s_i(x,y)}} \dfrac{dydx}{|x-y|^{N+s_i(x,y)}}\\
   &= \int_{\R^N\setminus \Omega} \left(  \int_{\Omega}  a^i_{(x,y)}\left( \dfrac{|u(x)-u(y)|}{|x-y|^{s_i(x,y)} }\right)\dfrac{u(x)-u(y)}{|x-y|^{s_i(x,y)}} \dfrac{dx}{|x-y|^{N+s_i(x,y)}}\right) dy\\
   &= -\int_{\R^N\setminus \Omega}\mathcal{N}^{s_i(x,\cdot)}_{a^i(x,\cdot)}u(y)dy.
   \end{aligned} 
   $$
  \end{proof} 
    \begin{prop}
  For all $u\in X_i$ $(i=1,2)$, we have
  $$ 
   \begin{aligned}
   \dfrac{1}{2} & \int_{\R^{2N}\setminus(C\Omega)^2}  a^i_{(x,y)}\left( \dfrac{|u(x)-u(y)|}{|x-y|^{s_i(x,y)} }\right)\dfrac{u(x)-u(y)}{|x-y|^{s_i(x,y)}}\dfrac{v(x)-v(y)}{|x-y|^{s_i(x,y)}} \dfrac{dxdy}{|x-y|^{N}}\\
   &=\int_\Omega v (-\Delta)^{s_i(x,\cdot)}_{a^i_{(x,\cdot)}}u dx+\int_{C \Omega} v \mathcal{N}^{s_i(x,\cdot)}_{a^i_{(x,\cdot)}}udx.\\
   \end{aligned}
   $$                     
                    \end{prop}        
    \begin{proof}
    By symmetric, and since $\R^{2N}\setminus(C\Omega)^2=(\Omega\times \R^N)\cup (C \Omega\times\Omega)$. Then, we have
\begin{equation}\label{N7}
  \begin{aligned}
  \dfrac{1}{2} & \int_{\R^{2N}\setminus(C\Omega)^2}  a^i_{(x,y)}\left( \dfrac{|u(x)-u(y)|}{|x-y|^{s_i(x,y)} }\right)\dfrac{u(x)-u(y)}{|x-y|^{s_i(x,y)}}\dfrac{v(x)-v(y)}{|x-y|^{s_i(x,y)}} \dfrac{dxdy}{|x-y|^{N}}\\
  &=\dfrac{1}{2}  \int_{\R^{2N}\setminus(C\Omega)^2} v(x) a^i_{(x,y)}\left( \dfrac{|u(x)-u(y)|}{|x-y|^{s_i(x,y)} }\right)\dfrac{u(x)-u(y)}{|x-y|^{s_i(x,y)}} \dfrac{dxdy}{|x-y|^{N+s_i(x,y)}}\\
  &~~- \dfrac{1}{2}  \int_{\R^{2N}\setminus(C\Omega)^2} v(y) a^i_{(x,y)}\left( \dfrac{|u(x)-u(y)|}{|x-y|^{s_i(x,y)} }\right)\dfrac{u(x)-u(y)}{|x-y|^{s_i(x,y)}} \dfrac{dxdy}{|x-y|^{N+s_i(x,y)}}\\
  &=\dfrac{1}{2}  \int_{\R^{2N}\setminus(C\Omega)^2} v(x) a^i_{(x,y)}\left( \dfrac{|u(x)-u(y)|}{|x-y|^{s_i(x,y)} }\right)\dfrac{u(x)-u(y)}{|x-y|^{s_i(x,y)}} \dfrac{dxdy}{|x-y|^{N+s_i(x,y)}}\\
    &~~+ \dfrac{1}{2}  \int_{\R^{2N}\setminus(C\Omega)^2} v(y) a^i_{(y,x)}\left( \dfrac{|u(y)-u(x)|}{|y-x|^{s_i(y,x)} }\right)\dfrac{u(y)-u(x)}{|y-x|^{s_i(y,x)}} \dfrac{dxdy}{|x-y|^{N+s_i(x,y)}}\\
    &= \int_{\R^{2N}\setminus(C\Omega)^2} v(x) a^i_{(x,y)}\left( \dfrac{|u(x)-u(y)|}{|x-y|^{s_i(x,y)} }\right)\dfrac{u(x)-u(y)}{|x-y|^{s_i(x,y)}} \dfrac{dxdy}{|x-y|^{N+s_i(x,y)}}\\
    &=\int_{\Omega}v(x)\int_{\R^{N}} a^i_{(x,y)}\left( \dfrac{|u(x)-u(y)|}{|x-y|^{s_i(x,y)} }\right)\dfrac{u(x)-u(y)}{|x-y|^{s_i(x,y)}} \dfrac{dxdy}{|x-y|^{N+s_i(x,y)}}\\
    &~~ +\int_{C\Omega} v(x) \int_{\Omega} a^i_{(x,y)}\left( \dfrac{|u(x)-u(y)|}{|x-y|^{s_i(x,y)} }\right)\dfrac{u(x)-u(y)}{|x-y|^{s_i(x,y)}} \dfrac{dxdy}{|x-y|^{N+s_i(x,y)}}\\
    &= \int_\Omega v (-\Delta)^{s_i(x,\cdot)}_{a^i_{(x,\cdot)}}u dx+\int_{C \Omega} v \mathcal{N}^{s_i(x,\cdot)}_{a^i(x,\cdot)}udx.\\
    \end{aligned}
\end{equation}
                       \end{proof}   
 However, the natural solution space to study  Problem \hyperref[P]{$(\mathcal{P}_{a})$} is given by                                             
 $$X=X_1\cap X_2$$ 
 endowed with the norm
 $$||u||_X=\|u\|_{X_1}+\|u\|_{X_2}$$                      Clearly $X$ is still a reflexive and separable Banach space with respect to $\|u\|_{X}$. It is not difficult to see that we can make
use of another norm on $X$ equivalent to $\|u\|_{X}$, given as                               
            $$
            \|u\|:=\|u\|_X=\inf \left\{\eta \geq 0: \rho_{s(x,y)}\left(\frac{u}{\eta}\right) \leq 1\right\}
            $$
            where the combined modular $\rho_{s(x,y)}: X \rightarrow \mathbb{R}$ is defined as
            $$
            \rho_{s(x,y)}(u)=\rho_{s_1(x,y)}(u)+\rho_{s_2(x,y)}(u),
            $$
            such that $\rho_{s_1(x,y)}$ and  $\rho_{s_2(x,y)}$ are described as in (\ref{smodN}).
            Arguing similarly to Proposition \ref{Nmod}, we can get the following comparison result.
   \begin{prop}\label{Nmod.}
               Assume that (\ref{v1}) is satisfied. Then, for any $u \in X$, the following relations hold true:
                 \begin{equation}\label{Nmod1.}
           ||u||>1\Longrightarrow      ||u||^{\min\left\lbrace \varphi_1^-, \varphi_2^-\right\rbrace } \leqslant  \rho(u)\leqslant  ||u||^{\max\left\lbrace \varphi_1^+, \varphi_2^+\right\rbrace },
                 \end{equation}
                 \begin{equation}\label{Nmod2.}
                      ||u||<1\Longrightarrow    ||u||^{\max\left\lbrace \varphi_1^+, \varphi_2^+\right\rbrace} \leqslant  \rho(u)\leqslant  ||u||^{\min\left\lbrace \varphi_1^-, \varphi_2^-\right\rbrace}. \end{equation}
                 \end{prop}

Based on the integration by part formula, we are now in position to state the natural definition of a weak solution of \hyperref[P]{$(\mathcal{P}_{a})$}. First, to simplify the notation, for arbitrary function $u, v\in X$, we set
{\small$$
\begin{aligned}
\mathcal{A}_{s}(u,v)   = & \sum_{i=1}^{2}\Big(  \dfrac{1}{2}  \int_{\R^{2N}\setminus(C\Omega)^2}  a^i_{(x,y)}\left( \dfrac{|u(x)-u(y)|}{|x-y|^{s_i(x,y)} }\right)\dfrac{u(x)-u(y)}{|x-y|^{s_i(x,y)}}\dfrac{v(x)-v(y)}{|x-y|^{s_i(x,y)}} \dfrac{dxdy}{|x-y|^{N}}\\
&+\int_{\Omega}\widehat{a}^i_{x}(|u|)u vdx+\int_{C \Omega}\beta(x) \widehat{a}^i_{x}(|u|)u vdx \Big) .
\end{aligned}
$$}       
We  say that $u\in X$ is a weak solution of \hyperref[P]{$(\mathcal{P}_{a})$} is
\begin{equation}\label{N6}
\mathcal{A}_{s}(u,v) =\lambda\int_\Omega f(x,u)vdx
\end{equation}
for all $v\in X$. 
\begin{rem}
Let us first state the definition of a weak solution to our problem $(\ref{N6})$. Note
that here we are using that $a_{x,y}$ is symmetric.  Therefore, In \cite{benkirane,benkirane2}, the authors must set the condition (\ref{n4}), to be the definition of weak solution has a meaning.
\end{rem}

As a consequence of this definition \ref{N6}, we have the following result.
\begin{prop}
Let $u\in X$ be a weak solution of \hyperref[P]{$(\mathcal{P}_{a})$}. Then
$$\sum_{i=1}^{2}\left( \mathcal{N}^{s_i(x,\cdot)}_{a^i_{(x,\cdot)}}u+\beta(x) \widehat{a}^i_{x}(|u|)u\right) =0~~\text{ a.e in   } \R^N\setminus \Omega.$$
\end{prop}
\begin{proof}
First, we take $v\in X$ such that $v=0$ in $\Omega$ as a test function in $(\ref{N6})$, and similar calculus to $(\ref{N7})$.  We have 
$$
\begin{aligned}
0=&\mathcal{A}_{s}(u,v)\\
&= \sum_{i=1}^{2} \Big( \dfrac{1}{2}  \int_{\R^{2N}\setminus(C\Omega)^2}  a^i_{(x,y)}\left( \dfrac{|u(x)-u(y)|}{|x-y|^{s_i(x,y)} }\right)\dfrac{u(x)-u(y)}{|x-y|^{s_i(x,y)}}\dfrac{v(x)-v(y)}{|x-y|^{s_i(x,y)}} \dfrac{dxdy}{|x-y|^{N}}\\
&~~+\int_{C \Omega}\beta(x) \widehat{a}^i_{x}(|u|)u vdx\Big)\\
&= \sum_{i=1}^{2} \Big( \int_{\Omega}  \int_{\R^N\setminus\Omega} a^i_{(x,y)}\left( \dfrac{|u(x)-u(y)|}{|x-y|^{s_i(x,y)} }\right)\dfrac{u(x)-u(y)}{|x-y|^{s_i(x,y)}}v(x) \dfrac{dxdy}{|x-y|^{N+s_i(x,y)}}\\
&~~+\int_{C \Omega}\beta(x) \widehat{a}^i_{x}(|u|)u vdx\Big)\\
&= \sum_{i=1}^{2} \Big( \int_{\R^N\setminus\Omega}v(x)\int_{\Omega}   a^i_{(x,y)}\left( \dfrac{|u(x)-u(y)|}{|x-y|^{s_i(x,y)} }\right)\dfrac{u(x)-u(y)}{|x-y|^{s_i(x,y)}} \dfrac{dxdy}{|x-y|^{N+s_i(x,y)}}\\
&~~+\int_{C \Omega}\beta(x) \widehat{a}^i_{x}(|u|)u vdx\Big)\\
&= \sum_{i=1}^{2} \Big( \int_{\R^N\setminus\Omega}v(x)\mathcal{N}^{s_i(x,\cdot)}_{a^i_{(x,\cdot)}}u(x)dx+\int_{C \Omega}\beta(x) \widehat{a}^i_{x}(|u|)u vdx\Big)\\
&= \sum_{i=1}^{2} \Big( \int_{\R^N\setminus\Omega}\left( \mathcal{N}^{s_i(x,\cdot)}_{a^i_{(x,\cdot)}}u(x)dx+\beta(x) \widehat{a}^i_{x}(|u|)u\right)  v(x) dx\Big).\\
\end{aligned} 
$$
This implies that 
$$ \sum_{i=1}^{2} \left( \int_{\R^N\setminus\Omega}\left( \mathcal{N}^{s_i(x,\cdot)}_{a^i_{(x,\cdot)}}u(x)dx+\beta(x) \widehat{a}^i_{x}(|u|)u\right)  v(x) dx\right) =0$$
for any $v\in X$, and $v=0$ in $\Omega$. In particular is true for every $v\in C^\infty_c(\R^N\setminus \Omega)$, and so
$$ \sum_{i=1}^{2} \left(  \mathcal{N}^{s_i(x,\cdot)}_{a^i_{(x,\cdot)}}u+\beta(x) \widehat{a}^i_{x}(|u|)u\right) =0~~\text{ a.e in   } \R^N\setminus \Omega.$$
\end{proof}
    \section{Existence results and proofs}\label{S4}
  The aim of this section is  to prove the existence of a weak solution of \hyperref[P]{$(\mathcal{P}_{a})$}. 
 In what follows, we will work with the modular norm $\|.\|$ and we denote by  $\left( X^*, ||.||_*\right)$         the dual space of $\left( X, ||.||\right)$.\\

    Next, we suppose that  $f : \Omega \times \R \rightarrow  \R$ is a Carath\'eodory function such that       
  \begin{equation}\label{f1}\tag{$f_1$}  |f(x,t)|\leqslant c_1|t|^{q(x)-1},  \end{equation}    
                \begin{equation}\label{f2}\tag{$f_2$} c_2|t|^{q(x)}\leqslant F(x,t):=\int_{0}^{t}f(x,\tau)d\tau,  \end{equation}
for all $x\in \Omega$ and all $t\in \R^N$,  where $c_1$ and $c_2$ are two positive constants, and $q\in C(\overline{\Omega})$ with $1<q^+\leqslant\min\left\lbrace  \varphi_1^- , \varphi_2^- \right\rbrace$, and
\begin{equation}\label{8}
   \lim_{t\rightarrow \infty}\left(\sup\limits_{x\in \overline{\Omega}}\dfrac{|t|^{q^+}}{(\widehat{\varPhi^i}_{x,s_i^-})_*(kt)}\right)=0 ~~~~i=1,2~~~~\forall k>0.
   \end{equation}
\begin{rem}\label{rem1}
     By $(\ref{8})$, we can apply Proposition $\ref{N10}$  we obtain that $X_{i}$ is compactly embedded in $L^{q+}(\Omega)$ for $i=1,2$. That fact combined with the continuous embedding of $L^{q^+}(\Omega)$ in $L^{q(x)}(\Omega)$, ensures that $X_{i}$ is compactly embedded in $L^{q(x)}(\Omega)$ for $i=1,2$.  In  particular, the embedding
     $$X \hookrightarrow L^{q(x)}(\Omega)$$
     is compact.
 \end{rem} 
       For simplicity, we set
   \begin{equation}
   D^{s_i(x,y)}u:=\dfrac{u(x)-u(y)}{|x-y|^{s_i(x,y)}}. \label{h} 
    \end{equation}       
       
   Now, we are ready to state our  existence result. 
      \begin{thm}\label{2.1.}
       Assume $f$ satisfy  \hyperref[f1]{$(f_1)$} and \hyperref[f2]{$(f_2)$}. Then there exist $\lambda_*$ and $\lambda^*$, such that for any $\lambda\in(0,\lambda_*)\cup [\lambda^*,\infty)$, problem \hyperref[P]{$(\mathcal{P}_{a})$} has a nontrivial weak solutions.    \end{thm}
           
    For each $\lambda>0$, we define the energy  functional $J_\lambda :  X\longrightarrow \R$ by
    {\small\begin{equation}\label{14.}
    \begin{aligned}
    J_\lambda(u)=&\displaystyle \sum_{i=1}^{2}\Big( \dfrac{1}{2}  \int_{\R^{2N}\setminus(C\Omega)^2} \varPhi^i_{x,y}\left( \dfrac{ |u(x)- u(y)|}{|x-y|^{s_i(x,y)}}\right) \dfrac{dxdy}{|x-y|^N}+\int_{\Omega}\widehat{\varPhi^i}_x\left( |u(x)|\right) dx\\
    &+\int_{C \Omega}\beta(x) \widehat{\varPhi^i}_x\left( |u(x)|\right) dx\Big)-\lambda\int_{\Omega}F(x,u)dx.
        \end{aligned}
    \end{equation}}
         \begin{rem}
    We note that the functional $J_\lambda :  X\longrightarrow \R$ in $(\ref{14.})$ is well
    defined. Indeed, if $u\in X$, then, we have  $u \in  L^{q(x)}(\Omega)$. Hence, by the condition \hyperref[f1]{$(f_1)$},
    $$ |F(x,u)|\leqslant\int_{0}^{u}|f(x,t)|dt=c_1|u|^{q(x)}$$
    and thus, 
    $$\int_{\Omega}|F(x,u)|dx<\infty.$$
         \end{rem}
   We first  establish some basis properties of $J_\lambda$.
   \begin{prop}\label{prop1}
    Assume condition \hyperref[f1]{$(f_1)$} is satisfied. Then, for each $\lambda>0$, $J_\lambda\in C^1\left( X, \R \right)$ with the derivative given by 
  $$
  \begin{aligned}
  \left\langle J'_\lambda(u),v\right\rangle =&\sum_{i=1}^{2}\Big( \dfrac{1}{2}  \int_{\R^{2N}\setminus(C\Omega)^2} a^i_{x,y}(|D^{s_i(x,y)}u|)D^{s_i(x,y)}u D^{s_i(x,y)}vd\mu+\int_{\Omega}\widehat{a}^i_{x}(|u|)u vdx\\
  &+\int_{C \Omega}\beta(x) \widehat{a}^i_{x}(|u|)u vdx\Big)-\lambda\int_{\Omega}f(x,u)vdx
    \end{aligned}
  $$
   for all $u,v \in X$.\end{prop} 
   Proof of this Proposition is similar to \cite[Proposition 3.1]{benkirane}.\\
   
    Now, define the functionals $I_i : X\longrightarrow \R$ $i=1,2$ by 
   $$
   \begin{aligned}
       I_1(u)=&\displaystyle \sum_{i=1}^{2}\Big( \dfrac{1}{2}  \int_{\R^{2N}\setminus(C\Omega)^2} \varPhi^i_{x,y}\left( \dfrac{ |u(x)- u(y)|}{|x-y|^{s_i(x,y)}}\right) \dfrac{dxdy}{|x-y|^N}+\int_{\Omega}\widehat{\varPhi^i}_x\left( |u(x)|\right) dx\\
       &+\int_{C \Omega}\beta(x) \widehat{\varPhi^i}_x\left( |u(x)|\right) dx\Big).
           \end{aligned}$$
   and
   $$I_2(u)= \int_{\Omega}F(x,u)dx.$$
   \begin{prop}\label{lem3}
  The functional $J_\lambda$ is weakly lower semi continuous.
   \end{prop}  
    \begin{proof} First, note that $I_1$
       is lower semi-continuous in the weak topology of $X$. Indeed, 
        since $\varPhi_{x,y}$ is a convex function so $I_1$ is also convex. Then, let $\left\lbrace u_n\right\rbrace \subset X$ with $u_n\rightharpoonup u$ weakly in $ X$, then by convexity of $I_1$ we have 
     $$I_1(u_n)-I_1(u)\geqslant \left\langle I_1'(u),u_n-u\right\rangle,$$ 
          	and hence, we obtain $$I_1(u)\leqslant \liminf I_1(u_n),$$ that is, the map $I_1$
    is  weakly lower semi continuous.          
  On the other hand, since $I_2\in  C^1\left( X, \|.\|\right),$ we have
          $$\lim\limits_{n\rightarrow \infty}\int_{\Omega}F(x,u_n)dx=\int_{\Omega}F(x,u)dx.$$
         Thus, we find 
          $$J_\lambda(u)\leqslant \liminf J_\lambda(u_n).$$
          Therefore, $J_\lambda$ is weakly lower semi continuous and Proposition $\ref{lem3}$ is verified. \end{proof}
      \begin{lem}\label{4.4.}
        Assume that the sequence $\left\lbrace u_n\right\rbrace $ converges weakly to $u$ in $X$ and 
        \begin{equation}\label{35.}
        \limsup_{n\rightarrow \infty}\left\langle I_1'(u_n),u_n-u\right\rangle \leqslant 0.
        \end{equation}
        Then the sequence $\left\lbrace u_n\right\rbrace$  is convergence strongly to $u$ in $X$.
        \end{lem}
                         \begin{proof} Since $u_n$ converges weakly to $u$ in $X$, then $\left\lbrace ||u_n||\right\rbrace $ is a bounded sequence of real numbers. Then by Proposition $\ref{Nmod}$, we deduce that $\left\lbrace I_1(u_n)\right\rbrace$ is bounded. So for a subsequence, we deduce that, 
              $$I_1(u_n)\longrightarrow c.$$
              Or since $I_1$ is weak lower semi continuous, we get 
              $$I_1(u)\leqslant \liminf_{n\rightarrow \infty}I_1(u_n)=c.$$
     On the other hand, by the convexity of $I_1$, we have 
     $$I_1(u)\geqslant I_1(u_n)+\left\langle I_1'(u_n),u_n-u\right\rangle .$$   
     Next, by the hypothesis  $(\ref{35.})$, we conclude that $$I_1(u)=c.$$
     Since $\left\lbrace \dfrac{u_n+u}{2} \right\rbrace $ converges weakly to $u$ in $X$, so since $I_1$ is  sequentially weakly lower semicontinuous :
     \begin{equation}\label{32.}
     c=I_1(u)\leqslant \liminf_{n\rightarrow \infty}I_1\left( \dfrac{u_n+u}{2}\right). \end{equation}
      We assume by  contradiction that $\left\lbrace u_n\right\rbrace$ does not converge to $u$ in $X$. Hence,  there exist a subsequence of $\left\lbrace u_n\right\rbrace $, still denoted by $\left\lbrace u_n\right\rbrace $ and there exits $\varepsilon_0>0$ such that 
        $$\Bigg|\Bigg|\dfrac{u_n-u}{2}\Bigg|\Bigg|\geqslant \dfrac{\varepsilon_0}{2},$$
        by Proposition $\ref{Nmod}$, we have
        $$I_1\left( \dfrac{u_n-u}{2}\right) \geqslant \varepsilon_0^{\varphi_1^\pm}+\varepsilon_0^{\varphi_2^\pm}.$$
        On the other hand, by the conditions (\ref{v1}) and (\ref{f2.}), we can apply \cite[Lemma 2.1]{Lam}  in order to obtain         \begin{equation}\label{33.}
        \dfrac{1}{2}I_1(u_n)+\dfrac{1}{2}I_1(u)-I_1\left( \dfrac{u_n+u}{2}\right) \geqslant I_1\left( \dfrac{u_n-u}{2}\right) \geqslant \varepsilon_0^{\varphi_1^\pm}+\varepsilon_0^{\varphi_2^\pm}.
        \end{equation}
        It follows from $(\ref{33.})$ that
        \begin{equation}\label{34.}
        I_1(u)-(\varepsilon_0^{\varphi_1^\pm}+\varepsilon_0^{\varphi_2^\pm}) \geqslant \limsup_{n\rightarrow \infty}I_1\left( \dfrac{u_n+u}{2}\right),
        \end{equation}
        from $(\ref{32.})$ and $(\ref{34.})$ we obtain a contradiction. This shows that $\left\lbrace u_n\right\rbrace$ converges strongly to $u$ in $X$.
           \end{proof}
   \begin{lem}\label{lem5}
   Assume the hypotheses of Theorem $\ref{2.1.}$  are fulfilled. Then there exist $\rho, \alpha>0$ and $\lambda_*>0$ such that for any $\lambda\in (0,\lambda_*),~~J_\lambda(u)\geqslant \alpha>0$ for any $u\in X$ with $||u||=\rho$.
   \end{lem}   
                \begin{proof}  Since $X$ is continuously embedded in $L^{q(x)}(\Omega)$. Then there exists a positive constant $c>0$ such that 
     \begin{equation}\label{28}
     ||u||_{q(x)}\leqslant c||u|| ~~\forall u\in X.
     \end{equation}   
    We fix $\rho\in (0,1)$ such that $\rho<\dfrac{1}{c}$. Then relation $(\ref{28})$ implies that for any $u\in X$ with $||u||=\rho$ : 
    $$\begin{aligned}
    J_\lambda(u)&\geqslant ||u||^{\max\left\lbrace \varphi_1^+, \varphi_2^+\right\rbrace}-\lambda c_2 c^{q^\pm}||u||^{q^\pm}\\
    &=\rho^{q^\pm}\left( \rho^{\max\left\lbrace \varphi_1^+, \varphi_2^+\right\rbrace-{q^\pm}}-\lambda c^{q^\pm} c_2\right).
    \end{aligned} 
     $$
   By the above inequality, we remark if we define 
   \begin{equation}\label{29}
   \lambda_*=\dfrac{\rho^{\max\left\lbrace \varphi_1^+, \varphi_2^+\right\rbrace-{q^\pm}}}{2c_2 c^{q^\pm}}.
   \end{equation}    
     Then for any $u\in X$ with $||u||=\rho$, there exists $\alpha=\dfrac{\rho^{\max\left\lbrace \varphi_1^+, \varphi_2^+\right\rbrace}}{2}>0$ such that 
     $$J_\lambda(u)\geqslant \alpha>0,~~\forall \lambda\in (0,\lambda_*).$$
     The proof of Lemma $\ref{lem5}$ is complete.           \end{proof}
    \begin{lem}\label{lem6}
    Assume the hypotheses of Theorem $\ref{2.1.}$  are fulfilled. Then there exists $\theta\in X$ such that $\theta>0$ and $J_\lambda(t\theta)<0$ for $t>0$ small enough.
    \end{lem} 
      \begin{proof} Let $\Omega_0\subset \subset \Omega$, for $x_0\in \Omega_0$, $0 < R < 1$ satisfy $B_{2R}(x_0)\subset \Omega_0$, where $B_{2R}(x_0)$
               is the ball of radius $2R$ with center at the point $x_0$ in $\R^N$. Let $\theta\in C_0^\infty(B_{2R}(x_0))$ satisfies $0\leqslant \theta \leqslant 1$ and
               $\theta \equiv 1$  in $B_{2R}(x_0)$. Theorem $\ref{TT}$ implies that $||\theta||<\infty.$ Then for $0 < t < 1$, by  $\hyperref[f2]{(f_2)}$, we have    
               $$
                  \begin{aligned}
                         J_\lambda(t\theta)=&\sum_{i=1}^{2}\Big(
                          \displaystyle\dfrac{1}{2} \int_{\R^{2N}\setminus(C\Omega)^2} \varPhi^i_{x,y}\left( \dfrac{|t\theta(x)-t\theta(y)|}{|x-y|^{s_i(x,y)} }\right) \dfrac{dxdy}{|x-y|^N} +\int_{\Omega}\varPhi^i_x(|t \theta|)dx\\
                          &+\int_{C \Omega}\beta(x) \widehat{\varPhi}^i_{x}(|t\theta|)dx \Big) -\lambda\int_{\Omega}F(x,t\theta)dx\\
                         &\leqslant  ||t\theta||^{\min\left\lbrace \varphi_1^-, \varphi_2^-\right\rbrace}-\lambda c_2\int_{\Omega_0} |t\theta|^{q(x)}dx\\
                         &\leqslant t^{\min\left\lbrace \varphi_1^-, \varphi_2^-\right\rbrace}||\theta||^{\min\left\lbrace \varphi_1^-, \varphi_2^-\right\rbrace}-\lambda c_2t^{q\pm}\int_{\Omega_0} |\theta|^{q(x)}dx.
                        \end{aligned}
                        $$
               Since $\min\left\lbrace \varphi_1^-, \varphi_2^-\right\rbrace > q^+$ and $\displaystyle\int_{\Omega_0} |\theta|^{q(x)}dx>0$ we have $J_\lambda(t_0\theta)<0$ for
               $t_0\in(0,t)$ sufficiently small.
                       \end{proof}    
   \begin{lem}\label{lem7}
   Assume the hypotheses of Theorem $\ref{2.1.}$  are fulfilled. Then for any $\lambda>0$ the functional $J_\lambda$ is coercive.
   \end{lem} 
   \begin{proof} For each $u\in X$ with $||u||>1$ and $\lambda>0$, relations $(\ref{Nmod1.})$,  $(\ref{28})$ and the condition \hyperref[f1]{$(f_1)$} imply
      $$
                   \begin{aligned}
                          J_\lambda(u)=& \sum_{i=1}^{2}\Big(\displaystyle\dfrac{1}{2}  \int_{\R^{2N}\setminus(C\Omega)^2} \varPhi^i_{x,y}\left( \dfrac{ |u(x)- u(y)|}{|x-y|^{s_i(x,y)}}\right) \dfrac{dxdy}{|x-y|^N}+\int_{\Omega}\widehat{\varPhi}^i_x\left( |u(x)|\right) dx\\
                                 &+\int_{C \Omega}\beta(x) \widehat{\varPhi}^i_x\left( |u(x)|\right) dx\Big)\\
                          &\geqslant  ||u||^{\min\left\lbrace \varphi_1^-, \varphi_2^-\right\rbrace}-\lambda c_1\int_{\Omega} |u|^{q(x)}dx\\
                          &\geqslant ||u||^{\min\left\lbrace \varphi_1^-, \varphi_2^-\right\rbrace}-\lambda c_1c||u||^{q^\pm}.
                         \end{aligned}
                         $$
    Since $\min\left\lbrace \varphi_1^-, \varphi_2^-\right\rbrace>q^+$ the above inequality   implies that $J_\lambda(u)\longrightarrow \infty$ as $||u||\rightarrow \infty$, that is, $J_\lambda$ is coercive.  \end{proof}
         \begin{proof}[\noindent \textbf{Proof of Theorem $\ref{2.1.}$}] Let $\lambda_*>0$ be defined as in $(\ref{29})$ and $\lambda\in (0,\lambda_*)$. By Lemma $\ref{lem5}$ it follows that on the boundary oh the ball centered in the origin and of radius $\rho$ in $X$, denoted by $B_\rho(0)$, we have 
         $$\inf\limits_{\partial B_\rho(0)}J_\lambda>0.$$
      On the other hand, by Lemma $\ref{lem6}$, there exists $\theta \in X$ such that $J_\lambda(t\theta)<0$ for all $t>0$ small enough. Moreover for any $u\in B_\rho(0)$, we have 
      $$
                      \begin{aligned}
                             J_\lambda(u)\geqslant ||u||^{\max\left\lbrace \varphi_1^+, \varphi_2^+\right\rbrace}-\lambda c_1c||u||^q.
                            \end{aligned}
                            $$
      It follows that
      $$-\infty<c:=\inf\limits_{\overline{B_\rho(0)}} J_\lambda<0.$$   
      We let now $0<\varepsilon <\inf\limits_{\partial  B_\rho(0)}  J_\lambda -  \inf\limits_{B_\rho(0)} J_\lambda.$    Applying Theorem $\ref{th1}$ to the functional 
      $J_\lambda : \overline{B_\rho(0)}\longrightarrow \R$, we find $u_\varepsilon \in \overline{B_\rho(0)}$ such that 
       $$
            \left\{ 
                 \begin{array}{clclc}
               J_\lambda(u_\varepsilon)&<\inf\limits_{\overline{B_\rho(0)}} J_\lambda+\varepsilon,& \\\\
                 J_\lambda(u_\varepsilon)&< J_\lambda(u)+\varepsilon ||u-u_\varepsilon||,& \text{  } u\neq u_\varepsilon.
                 \end{array}
                 \right. 
              $$
       Since  $J_\lambda(u_\varepsilon)\leqslant  \inf\limits_{\overline{B_\rho(0)}} J_\lambda+\varepsilon\leqslant \inf\limits_{B_\rho(0)} J_\lambda+\varepsilon \leqslant \inf\limits_{\partial  B_\rho(0)}  J_\lambda$, we deduce $u_\varepsilon  \in B_\rho(0)$. 
       
       Now, we define $\Lambda_\lambda :  \overline{B_\rho(0)}\longrightarrow \R$ by 
       $$\Lambda_\lambda(u)=J_\lambda(u)+\varepsilon||u-u_\varepsilon||.$$
       It's clear that $u_\varepsilon$ is a minimum point of $\Lambda_\lambda$ and then
       $$\dfrac{\Lambda_\lambda(u_\varepsilon+t v)-\Lambda_\lambda(u_\varepsilon)}{t}\geqslant 0$$ 
       for small $t>0$, and any $v\in B_\rho(0).$ The above relation yields 
           $$\dfrac{J_\lambda(u_\varepsilon+t v)-J_\lambda(u_\varepsilon)}{t}+\varepsilon||v||\geqslant 0.$$
           Letting $t\rightarrow 0$ it follows that $\left\langle J'_{\lambda}(u_\varepsilon),v\right\rangle +\varepsilon ||v||>0$ and we infer that $||J'_{\lambda}(u_\varepsilon)||_*\leqslant \varepsilon$.
            We deduce that there exists a sequence $\left\lbrace v_n\right\rbrace \subset B_\rho(0)$ such that 
            \begin{equation}\label{10}
            J_\lambda(v_n) \longrightarrow c \text{ and } J'_\lambda(v_n)\longrightarrow 0.
            \end{equation}
         It is clear that $\left\lbrace v_n\right\rbrace $ is bounded in $X$. Thus, there exists $v\in X$, such that up to a subsequence   $\left\lbrace v_n\right\rbrace $ converges weakly to $v$ in $X$. Since  $X$ is a compactly embedded in $L^{q(x)}(\Omega)$. The above information combined with condition \hyperref[f1]{$(f_1)$} and   H\"older's inequality implies
       \begin{equation}\label{11}
   \begin{aligned}
         \left| \int_{\Omega} f(x,v_n)(v_n-v)dx\right| &\leqslant c_1\int_{\Omega} \left| v_n\right| ^{q(x)-1}\left| v_n-v\right| dx\\   
         &\leqslant c_1\left| \left| |v_n|^{q(x)-1}\right| \right|_{\frac{q(x)}{q(x)-1}}\left| \left| v_n-v\right| \right|_{q(x)} \longrightarrow 0.
         \end{aligned} 
            \end{equation}
    On the other hand, by $(\ref{10})$ we have 
    \begin{equation}\label{12}
    \lim\limits_{n\rightarrow \infty}\left\langle J'_\lambda(v_n) , v_n-v\right\rangle =0.
    \end{equation}  
    Relations $(\ref{11})$ and $(\ref{12})$ imply 
    $$
     \lim\limits_{n\rightarrow \infty}\left\langle I'_1(v_n) , v_n-v\right\rangle =0.$$
     Thus, by Lemma $\ref{4.4.}$ we find that  $\left\lbrace v_n\right\rbrace $ converges strongly to $v$ in $X$, so by $(\ref{10})$: 
     $$ J_\lambda(v)=c<0 \text{ and }  J'_\lambda(v)=0.$$
   We conclude that $v$ is a nontrivial weak solution for problem   \hyperref[P]{$(\mathcal{P}_{a})$} for any $\lambda\in(0,\lambda_*)$.
   
   Next, by Lemma $\ref{lem7}$ and Proposition $\ref{lem3}$ we infer that $J_\lambda$ is coercive and weakly lower semi continuous in $X$ for all $\lambda>0$. Then Theorem $\ref{th2}$ implies that there exists $u_\lambda \in X$ a global minimized of $J_\lambda$ and thus a weak solution of problem \hyperref[P]{$(\mathcal{P}_{a})$}.
   
   Now, we show that $u_\lambda$ is non trivial. Indeed, letting $t_0>1$ be fixed real and
   $$
               \left\{ 
                    \begin{array}{clclc}
                  u_0(x)&=t_0~~\text{ in }  \Omega & \\\\
                    u_0(x)&=0~~\text{ in }  \R^N\setminus \Omega ,&
                    \end{array}
                    \right. 
                 $$
    we have $ u_0\in X$ and 
   $$
   \begin{aligned}
   J_\lambda(u_0)&=I_1(u_0)-\lambda\int_\Omega F(x,u_0)dx\\
   &=\sum_{i=1}^{2}\Big(\int_\Omega\widehat{\varPhi}^i_x(t_0)dx\Big)-\lambda\int_\Omega F(x,t_0)dx\\
   &\leqslant \sum_{i=1}^{2}\Big(\int_\Omega \widehat{\varPhi}^i_x(t_0)dx\Big)-\lambda c_2\int_\Omega |t_0|^{q(x)} dx\\
   &=L-\lambda c_2|t_0|^{q^-}|\Omega|,
   \end{aligned}
   $$
   where $L$ is a positive constant. Thus, for $\lambda^*>0$ large enough, $J_\lambda(u_0)<0$ for any $\lambda\in [\lambda^*,\infty)$. It follows that $J_\lambda(u_\lambda)<0$ for any $\lambda\in [\lambda^*,\infty)$ and thus $u_\lambda$ is a nontrivial weak solution of problem  \hyperref[P]{$(\mathcal{P}_{a})$} for any $\lambda\in [\lambda^*,\infty)$. Therefore, problem  \hyperref[P]{$(\mathcal{P}_{a})$} has a nontrivial weak solution for all $\lambda\in(0,\lambda_*)\cup [\lambda^*,\infty)$.\end{proof}
      



\begin{thebibliography}{}
%
%
  \bibitem{benkirane}    E. Azroul , A. Benkirane, M. Shimi and M. Srati (2020): \textit{On a class of
  nonlocal problems in new fractional Musielak-Sobolev spaces}, \textit{Applicable Analysis}, doi:
  10.1080/00036811.2020.1789601.
  
   \bibitem{benkirane2}    E. Azroul , A. Benkirane , M. Shimi and M. Srati (2020): \textit{Embedding and extension results in fractional Musielak–Sobolev spaces, Applicable
   Analysis}, \textit{Applicable Analysis}, doi:
   10.1080/00036811.2021.1948019.

  
   
  \bibitem{3} E. Azroul, A. Benkirane, M. Srati, \textit{Existence of solutions for a nonlocal type problem in fractional Orlicz Sobolev spaces}, \textit{Adv. Oper. Theory}  (2020) doi: 10.1007/s43036-020-00042-0.
  

     \bibitem{sr5} E. Azroul, A. Benkirane, M. Srati, \textit{Nonlocal eigenvalue type problem in fractional Orlicz-Sobolev space}, \textit{Adv. Oper. Theory} (2020) doi: 10.1007/s43036-020-00067-5.

     
        

 \bibitem{SRN1} E. Azroul, A. Benkirane, M. Srati, \textit{Eigenvalue problem associated with nonhomogeneous integro-differential operators}  \textit{J. Elliptic Parabol Equ} (2021). https://doi.org/10.1007/s41808-020-00092-8.
     
     \bibitem{SRN2}    E. Azroul, A. Benkirane and M. Srati, \textit{Mountain pass type solutions for a nonlacal fractional $a(.)$-Kirchhoff type problems}, \textit{Journal of Nonlinear Functional Analysis}, Vol. 2021 (2021), Article ID 3, pp. 1-18.
     
        \bibitem{SRT}    E. Azroul, A. Benkirane, M. Srati, and C. Torres, \textit{Infinitely many solutions for a nonlocal type problem with sign-changing weight function}. \textit{ Electron. J. Differential Equations}, Vol. \textbf{2021} (2021), No. 16, pp. 1-15. 
            
 
   
    \bibitem{SRH}    E. Azroul, A. Benkirane, M. Shimi and M. Srati, \textit{On a class of fractional $p(x)$-Kirchhoff type
     problems}. \textit{
     Applicable Analysis} (2019) doi: 10.1080/00036811.2019.1603372.
     

\bibitem{kamali} E. Azroul, N. Kamali, M. A. Ragusa, M. Shimi, \textit{On a p(x,.)-integrodifferential problem with Neumann boundary conditions}. \textit{Z. Anal. Anwend}. (2024), 
doi : 10.4171/ZAA/1780?    

    
  \bibitem{N6} A. Bahrouni, V.  Radulesc\u{u}, and P. Winkert, \textit{Robin fractional problems with symmetric
   variable growth}, J. Math. Phys. 61, 101503 (2020); doi: 10.1063/5.0014915.
   
 \bibitem{N7}   S. Bahrouni and A. Salort \textit{Neumann and Robin type boundary conditions in Fractional Orlicz-Sobolev
   spaces}  (2021)  \textit{ESAIM Control Optimisation and Calculus of Variations}	doi : 10.1051/cocv/2020064
    
     \bibitem{N1} G. Barles, E. Chasseigne,  C. Georgelin,  and  E. R. Jakobsen,  \textit{On Neumann
    type problems for nonlocal equations in a half space}. Trans. Amer. Math. Soc. 366
    (2014), no. 9, 4873-4917.
    
    
    \bibitem{N2} G. Barles, C. Georgelin,  and  E. R. Jakobsen,  \textit{On Neumann and oblique
    derivatives boundary conditions for nonlocal elliptic equations}. J. Differential Equations 256 (2014), no. 4, 1368–1394.
    
  

               
                  
               \bibitem{s1} Biswas R, Tiwari S. \textit{Multiplicity and uniform estimate for a class of variable order fractional $p(x)$-Laplacian problems with concave-convex
                   nonlinearities}. arXiv:1810.12960.  

\bibitem{x1} R. Biswas,  S. Bahrouni, \&  M.L. Carvalho, Fractional double phase Robin problem involving variable order-exponents without Ambrosetti–Rabinowitz condition. Z. Angew. Math. Phys. 73, 99 (2022). doi : 10.1007/s00033-022-01724-w

           
  
  
\bibitem{N3} C.  Cortazar, M. Elgueta, J. D. Rossi,   and N. Wolanski,  \textit{How to approximate the heat equation with Neumann boundary conditions by nonlocal diffusion
problems}. Arch. Ration. Mech. Anal. 187 (2008), no. 1, 137-156.

 
    \bibitem{N5}  S. Dipierro, X. Ros-Oton, and E. Valdinoci, \textit{Nonlocal problems with Neumann boundary conditions}, Rev. Mat. Iberoam. 33(2), 377–416 (2017)   
         

 
\bibitem{ek}  I. Ekeland
          On the variational principle
          J. Math. Anal. Appl., 47 (1974), pp. 324-353. 
          
          
     
           
            

  
  
 \bibitem{Lam} J. Lamperti, \textit{On the isometries of certain function-spaces}, Pacific J. Math. \textbf{8} (1958),
           459-466. 
  
   \bibitem{ra} M. Mih\"ailescu, V. R\"adulescu, Neumann problems associated to nonhomogeneous differential operators in Orlicz-Soboliv spaces, Ann. Inst. Fourier 58 (6) (2008) 2087-2111.
   
  \bibitem{N4}  D. Mugnai and E. Proietti Lippi, \textit{Neumann fractional p-Laplacian: Eigenvalues and existence results}, Nonlinear Anal. 188, 455-474 (2019).
  
    \bibitem{mu} J. Musielak; Orlicz Spaces and Modular Spaces, Lecture Notes in Mathematics, Vol. 1034,
   Springer, Berlin, 1983.
   
  
  
  
      
       \bibitem{110}     M. Struwe, \textit{Variational Methods: Applications to Nonlinear Partial Differential Equations and Hamiltonian Systems},
                          Springer-Verlag, Berlin, Heidelberg, 1990.
                          
     \bibitem{srati}  M. Srati,  \textit{Eigenvalue type problem in $s(\cdot, \cdot)$-fractional Musielak–Sobolev spaces.} J Elliptic Parabol Equ 10, 387-413 (2024).  doi : 10.1007/s41808-024-00269-5. 
      \bibitem{srati2} M. Srati,  E. Azroul, and A. Benkirane, \textit{Nonlocal problems with Neumann and Robin boundary condition in fractional Musielak-Sobolev spaces}, \textit{Rend. Circ. Mat. Palermo, II. Ser} (2024). doi : 10.1007/s12215-024-01117-0.                
\end{thebibliography}
\end{document}